\documentclass[11pt,a4paper,final,abstracton,numbers=noenddot]{scrartcl}

\usepackage[utf8]{inputenc}

\usepackage{titling}
\usepackage{graphicx}
\usepackage[english]{babel}
\usepackage{datetime}
\usepackage{enumitem}
\newlist{ienum}{enumerate}{3}
\setlist[enumerate]{label=(\roman*),partopsep=0pt,itemsep=0pt,topsep=3pt,wide,leftmargin=0.2em,labelindent=0em}
\setlist[ienum]{label=(\roman*),partopsep=0pt,itemsep=0pt,topsep=3pt,wide,leftmargin=.5em,labelindent=.5em}
\setlist{nolistsep}
\usepackage{tikz}
\usetikzlibrary{matrix,arrows}
\usetikzlibrary{matrix,positioning}
\usetikzlibrary{decorations.markings}
\usetikzlibrary{calc}
\usetikzlibrary{shapes}
\usepackage{pgfplots}
    \pgfplotsset{
        compat=1.12,
    }
\usetikzlibrary{intersections}

\usepackage{fancyhdr}
\usepackage[style=alphabetic,citestyle=alphabetic]{biblatex}
\DeclareFieldFormat[article]{title}{\mkbibemph{#1\isdot}}
\defbibheading{bibliography}{\centering\huge\textbf{References}\vspace{1cm}}

\usepackage{showlabels,rotating}

\usepackage{fix-cm}
\usepackage{ifdraft}
\usepackage{showlabels,rotating}

\ifdraft{\usepackage{draftwatermark}}{}

\usepackage[pdftex,draft=false]{hyperref}
\hypersetup{%
pdftitle = { Integrality Properties in the Moduli Space of Elliptic Curves: CM Case  },
pdfsubject = { Number Theory },
pdfauthor = { Stefan Schmid },
pdfkeywords = { number theory, elliptic curves, singular moduli },
plainpages=true,
pdfborder=0 0 0
}

\usepackage{amssymb}
\usepackage{amsmath}
\usepackage{amsthm}

\newtheoremstyle{named}
{}
{}
{\itshape}
{}
{\bfseries}
{}
{\newline}
{\thmname{#1} \thmname{#2} (\textit{\thmname{#3}})}

\newtheoremstyle{mythm}
{}
{}
{\itshape}
{}
{\bfseries}
{}
{1em}
{\thmname{#1} \thmname{#2}}

\newtheoremstyle{futref}
{}
{}
{\itshape}
{}
{\bfseries}
{}
{1em}
{\thmname{#1} \thmname{#3}}

\newtheoremstyle{mydef}
{}
{}
{\itshape\let\emph\textbf}
{}
{\bfseries}
{.}
{1em}
{\thmname{#1}}

\newtheoremstyle{myex}
{}
{}
{}
{}
{\bfseries}
{}
{1em}
{\thmname{#1} \thmname{#2}}

\theoremstyle{futref}

\theoremstyle{mythm}
\newtheorem{theorem}{Theorem}
\newtheorem{prop}[theorem]{Proposition}
\newtheorem{lem}[theorem]{Lemma}
\newtheorem{cor}[theorem]{Corollary}
\theoremstyle{mydef}

\theoremstyle{myex}

\theoremstyle{named}

\theoremstyle{remark}

\let\oldproof=\proof
\let\oldendproof=\endproof
\renewenvironment{proof}{\oldproof[\textsc{\proofname}]}{\oldendproof}

\newcommand{\N}{\ensuremath{\mathbb{N}}}    
\newcommand{\Z}{\ensuremath{\mathbb{Z}}}    
\newcommand{\Q}{\ensuremath{\mathbb{Q}}}    
\renewcommand{\H}{\ensuremath{\mathbb{H}}}    
\newcommand{\C}{\ensuremath{\mathbb{C}}}    
\newcommand{\F}{\ensuremath{\mathcal{F}}}	
\newcommand{\Pen}{\ensuremath{\mathcal{P}}}	

\newcommand{\cm}{\ensuremath{\mathcal{C}}}	
\newcommand{\Quads}{\ensuremath{\mathcal{Q}}}	

\newenvironment{smatrix}{\bigl(\begin{smallmatrix}}{\end{smallmatrix}\bigr)}

\let\Re\relax
\DeclareMathOperator{\Re}{\operatorname{Re}}
\let\Im\relax
\DeclareMathOperator{\Im}{\operatorname{Im}}

\DeclareMathOperator{\sl2z}{\operatorname{SL_2(\Z)}}

\DeclareMathOperator{\gcdt}{\operatorname{gcd_2}}

\let\mod\bmod

\newcommand{\twostack}[2]{\subarray{c}\scriptscriptstyle#1\\\scriptscriptstyle#2\endsubarray}

\let\emph\textit
\let\subset\subseteq

\numberwithin{theorem}{section}

\title{Integrality properties in the Moduli Space of Elliptic Curves: CM Case}
\author{Stefan Schmid}
\date{}

\begin{document}

\setlist[enumerate]{label=(\roman*),partopsep=0pt,itemsep=0pt,topsep=3pt,wide,leftmargin=0.2em,labelindent=0em}
\setlist[ienum]{label=(\roman*),partopsep=0pt,itemsep=0pt,topsep=3pt,wide,leftmargin=.5em,labelindent=.5em}

\maketitle

\begin{abstract}
A result of Habegger \cite{habSM} shows that there are only finitely many singular moduli
such that $j$ or $j-\alpha$ is an algebraic unit. The result uses Duke's Equidistribution Theorem
and is thus not effective. For a fixed $j$--invariant $\alpha \in \bar{\Q}$ of an
elliptic curve without complex multiplication, we prove that there are only finitely
many singular moduli $j$ such that $j-\alpha$ is an algebraic unit.
The difference to \cite{habSM} is that we give explicit bounds.
\end{abstract}

\thispagestyle{empty}

\section{Introduction}
\newcommand{\T}{\ensuremath{\mathcal{T}}}

We denote Klein's modular function defined on the upper half--plane by $j$.
A singular modulus is the value of the $j$--function evaluated at
an imaginary quadratic $\tau$. Singular moduli correspond to  the $j$--invariants
of elliptic curves with complex multiplication.
A classical result by Kronecker states that singular moduli are algebraic integers.

In 2011, Masser asked at the AIM workshop on unlikely intersections in algebraic
groups and Shimura varieties in Pisa, if there are only finitely many singular moduli
that are algebraic units. His question was motivated by
\cite{bilumasserzannier} and a result on the Andr\'e--Oort conjecture.
In 2014, Habegger gave an answer in \cite{habSM} to this question by proving
that at most finitely many singular moduli are algebraic units.
No examples were known at that time.
He also proved that only finitely many singular moduli exist such that $j-\alpha$
is a unit. Here examples are known, e.g.~one can take $\alpha = 1$ and then $j=0$
is a solution.
His results in \cite{habSM} rely on Duke's Equidistribution Theorem and are thus
not effective.
In 2018 Bilu, Habegger and K\"uhne proved in \cite{bhk} that no singular
modulus is a unit.

In the work at hand we want to give effective bounds for the number of
singular moduli $j$ such that $j-\alpha$ is a unit, where $\alpha \in \bar\Q$
is the $j$--invariant of an elliptic curve without complex multiplication.
More specifically, we obtain a bound on the discriminant of the endomorphism ring
of the singular moduli satisfying the condition.
See a recent result of Li \cite{li2018singular} for differences of singular
moduli.

We fix a singular modulus $j$. Elliptic curves with $j$--invariant $j$ have
the same full endomorphism ring. Let $\Delta$ denote the discriminant of this
ring.
We want to prove the following result.

\begin{theorem}\label{thm:cm_finite}
		Let $j$ be a singular modulus and let $\Delta$ be its discriminant. Let $\alpha$ be an
	algebraic number that is the $j$--invariant of an elliptic curve without complex multiplication.
	If we assume that $j-\alpha$ is an algebraic unit, then $|\Delta|$ is bounded from above by
	an explicit constant.
	Thus there are only finitely many singular moduli $j$ such that $j-\alpha$ is an algebraic unit.
	The constant can be found on page \pageref{eq:delta_bound}.

\end{theorem}

The sketch of the proof is as
follows. Write $j(\xi) = \alpha \in \bar\Q$ where the elliptic curve
associated to $\alpha$ does not have complex multiplication.
We can assume that $\xi$ is in the classical fundamental domain $\F$
of the upper half--plane.
To a singular modulus $j$ we have attached an elliptic
curve with complex multiplication.
We can write $\Delta = Df^2$ where $f$ is the conductor
of the endomorphism ring in the full ring of integers of
$\Q(\sqrt{\Delta})$, and $D$ is the discrimiant of that field.
The Galois conjugates of $j$ form a full orbit of length 
the class number $\cm(\Delta)$.
We write $\cm(\Delta; \xi; \varepsilon)$ for the number of
singular moduli  with discriminant $\Delta$ that can be written in the form $j(\tau)$
with $\tau \in \F$ and such that $|\tau - \xi| < \varepsilon$.
We prove an explicit bound on $\cm(\Delta; \xi; \varepsilon)$
which is given by
\begin{equation}\label{eq:sm_bound}
	\cm(\Delta; \xi; \varepsilon)
	\le F(\Delta)\left(32 |\Delta|^{1/2} \varepsilon^2 \log\log(|\Delta|^{1/2})
			+ 11\vert\Delta\vert^{1/2}\varepsilon
			+ 2 \right)
\end{equation}
for $|\Delta| \ge 10^{14}$ and $0 < \varepsilon < 1/2$. Here
\begin{displaymath}
	F(\Delta) = \max\left\{2^{\omega(a)} ; a \le |\Delta|^{1/2} \right\},
\end{displaymath}
and $\omega(n)$ is the number of distinct prime divisors of $n$.

The idea is to give lower and upper bounds of the logarithmic (Weil) height
of $j-\alpha$ that contradict each other.
The height of an algebraic number basically measures its complexity.
Let $K$ be a number field containing $\beta \in \bar\Q$.
If $M_K$ is a set of representatives of non--trivial absolute values extending
the $p$--adic absolute values and the usual absolute value, and
$[K_\nu:\Q_\nu]$, $\nu \in M_K$, denotes the local degrees, then
the height of $\beta \in \bar\Q$ is defined by
\begin{displaymath}
	h(\beta) = \frac 1{[K:\Q]}\sum_{\nu \in M_K} [K_\nu:\Q_\nu]\log \max\{1,|\beta|_\nu\}.
\end{displaymath}
The definition does not depend on the number field $K$.
If $\beta$ is a unit in the ring of integers, then the height can be computed through
\begin{displaymath}
	h(\beta) = h(\beta^{-1}) = \frac 1{[K:Q]} \sum_{\lvert \sigma(\beta^{-1}) \rvert > 1} \log \lvert \sigma(\beta^{-1}) \rvert
		= -\frac 1{[K:Q]} \sum_{\lvert \sigma(\beta) \rvert < 1} \log \lvert \sigma(\beta) \rvert.
\end{displaymath}
where $\sigma$ runs over all field embeddings $\sigma\colon K \hookrightarrow \C$.
More information on heights can be found in \cite{Gubler}.

Now if $j-\alpha$ is an algebraic unit, the height can be bounded as
\begin{equation}\label{eq:j-a_bound}
	h(j-\alpha) \ll \frac{\cm(\Delta; \xi'; \varepsilon)}{\cm(\Delta)}(\log|\Delta|)^4
						- \log \varepsilon
\end{equation}
for some $\xi' \in \F$ associated to $\xi$.
The constant in the inequality depends on $\alpha$.
We put $E(\Delta) = F(\log|\Delta|)^4$ and roughly choose $\varepsilon$ to be
\begin{displaymath}
	\varepsilon = \frac{\cm(\Delta)}{E(\Delta)|\Delta|^{1/2}}.
\end{displaymath}
If we substitute this and \eqref{eq:sm_bound} into \eqref{eq:j-a_bound}
and use estimates for $\omega(n)$ by Robin \cite{robin1983estimation}
we get
\begin{equation}\label{eq:j-a_bound_fin}
	h(j-\alpha) \ll \frac{E(\Delta)}{\cm(\Delta)} + \log\frac{E(\Delta)|\Delta|^{1/2}}{\cm(\Delta)}.
\end{equation}
To bound $|\Delta|$ from above we need lower bounds for the
height of $j-\alpha$. One can prove
\begin{displaymath}
	h(j-\alpha) \gg \log|\Delta|
\end{displaymath}
and
\begin{displaymath}
	h(j-\alpha) \gg \frac{|\Delta|^{1/2}}{\cm(\Delta)}.
\end{displaymath}
The first inequality follows from work of Colmez \cite{colmez} and
Nakkajima--Taguchi \cite{nakkajimataguchi}, and the second inequality is more elementary.
Again the bounds depend on $\alpha$.
Combining the lower bounds with the upper bound from \eqref{eq:j-a_bound_fin}
we obtain
\begin{displaymath}
	\max\left\{\frac{|\Delta|^{1/2}}{\cm(\Delta)},\log|\Delta|\right\} \ll
			\frac{E(\Delta)}{\cm(\Delta)} + \log\frac{|\Delta|^{1/2}}{\cm(\Delta)} + \log E(\Delta)
\end{displaymath}
for large $|\Delta|$.
Further analysis shows that $E(\Delta)|\Delta|^{-1/2} = |\Delta|^{o(1)}$
and $\log E(\Delta)/\log |\Delta| = o(1)$. Thus the inequality can not
hold for large values of $|\Delta|$.
All constants in the above deductions can be made explicit, but
some are very large.

\section{Bounding points in the fundamental domain}

Let $\Quads(\Delta)$ be the set ($\subset \Z^3$) of coefficients representing reduced primitive,
positive definite quadratic forms with discriminant $\Delta$ and let $\cm(\Delta)$ be the
class number.
We will write $\Delta = Df^2$ throughout this exposition, where $D$ is the discriminant of the
imaginary quadratic field $\Q(\sqrt{\Delta})$ and $f \in \N$ is called the \emph{conductor}.
For $\xi \in \F$ and $\varepsilon > 0$ we define
\begin{displaymath}
	\cm(\Delta; \xi; \varepsilon) = \#\left\{ (a,b,c) \in \Quads(\Delta);
		\left\vert \frac{-b + \Delta^{1/2}}{2a} - \xi \right\vert < \varepsilon \right\}.
\end{displaymath}
We also define the function $F$ of $\Delta$ by
\begin{displaymath}
	F = F(\Delta) = \max\left\{2^{\omega(a)} ; a \le |\Delta|^{1/2} \right\},
\end{displaymath}
where $\omega(n)$ is the number of distinct prime divisors of $n$.
We also define the modified conductor by
\begin{displaymath}
	\tilde{f} = \begin{cases}f&D \equiv 1 \mod 4,\\2f&D\equiv 0\mod 4.
	\end{cases}
\end{displaymath}
Then $\Delta/\tilde f^2$ is square--free.

Let $\sigma_k(n) = \sum_{d|n}d^k$.
We are now ready to state the first lemma that gives a bound on the $\tau$ in a neighborhood of a
fixed point such that $j(\tau)$ is a singular modulus of fixed discriminant.
While this is a generalization of Theorem 2.1 in \cite{bhk}, the constants are not as good as
in that very paper.

\begin{lem}\label{lem:delta_count}
	Let $\Delta$ be a negative integer, $y = \Im(\xi) \ge \sqrt{3}/2$ and $0 < \varepsilon < 1/4$. Then
	\begin{displaymath}
		\cm(\Delta; \xi; \varepsilon)
			\le F(\Delta)\left(32\frac{\sigma_1\left(\tilde f\right)}{\tilde f} \frac{|\Delta|^{1/2}}{4y^2 - 1} \varepsilon^2
				+ 8 \sigma_0\left(\tilde f\right) \left\vert\frac{\Delta}3\right\vert^{1/4}\varepsilon
				+ 8\frac{|\Delta|^{1/2}}{4y^2 - 1} \varepsilon
				+ 2 \right)
	\end{displaymath}
\end{lem}

\begin{proof}
	We start with $|\tau - \xi| < \varepsilon$.
	This implies that the real and imaginary parts satisfy
	\begin{align*}
		\Im(\tau) &\in (\Im(\xi) - \varepsilon, \Im(\xi) + \varepsilon)	\\
		\Re(\tau) &\in (\Re(\xi) - \varepsilon, \Re(\xi) + \varepsilon).
	\end{align*}
	Now $\tau$ is of the form $(-b+\sqrt{\Delta})/2a$ and thus $\Im(\tau) = |\Delta|^{1/2}/2a$ and
	$\Re(\tau) = -b/2a$. This amounts to
	\begin{displaymath}
		y - \varepsilon < \frac{|\Delta|^{1/2}}{2a} < y + \varepsilon
	\end{displaymath}
	or equivalently
	\begin{displaymath}
		a \in \left( \frac{|\Delta|^{1/2}}{2y + 2\varepsilon}, \frac{|\Delta|^{1/2}}{2y - 2\varepsilon} \right) =: I.
	\end{displaymath}
	For $b$ we obtain
	\begin{equation}\label{eq:b_interval}
		2a(\Re(\xi)-\varepsilon) < -b < 2a(\Re(\xi)+\varepsilon),
	\end{equation}
	so $b$ lies in an interval of length $4a\varepsilon$.
	For two integers $m$ and $n$ we denote by $\gcdt(m,n)$ the greatest common divisor $d$ of $m$ and $n$
	such that $d^2|m$ and $d^2|n$.
	We have $\Delta = b^2 - 4ac$, so in particular $b^2 \equiv \Delta \mod a$.
	Thus, the residue classes modulo $a/\gcd_2(a,\Delta)$ of $b \in \Z$ satisfying $b^2 \equiv \Delta \mod a$
	is at most $2^{\omega(a/\gcd(a,\Delta))+1}$ by Lemma 2.4 in \cite{bhk}.
	Note that we have $\omega(a/\gcd(a,\Delta)) \le \omega(a)$.
	But $b$ also lies in the interval given in equation \eqref{eq:b_interval}, so that by
	Lemma 2.5 of \cite{bhk} there are at most
	\begin{equation}\label{eq:possible_b}
		\left(\frac{2a(\Re(\xi) + \varepsilon) - 2a(\Re(\xi) - \varepsilon)}{a/\gcdt(a,\Delta)} + 1\right)2^{\omega(a)+1}
		= (4\varepsilon \gcdt(a,\Delta) + 1)2^{\omega(a)+1}
	\end{equation}
	possible $b$'s for any fixed $a$.
	We have $a \le |\Delta/3|^{1/2}$ by Lemma 5.3.4 in \cite{cohen2013cant}, so that $2^{\omega(a)} \le F$.
	Using the equality in \eqref{eq:possible_b} and applying Lemma 5.3.4 of \cite{cohen2013cant} in the second inequality we get
	\begin{align*}
		\cm(\Delta; \xi; \varepsilon) &\le 8\varepsilon \sum_{a \in I\cap \Z} \gcdt(a,\Delta)2^{\omega(a)}
													+ 2\sum_{a \in I\cap \Z} 2^{\omega(a)}	\\
			&\le 8\varepsilon F \sum_{a \in I\cap \Z} \gcdt(a,\Delta) + 2F \#(I\cap\Z)	\\
			&\le 8\varepsilon F \sum_{\twostack{d^2 | \Delta}{d \le |\Delta/3|^{1/4}}} d\cdot \#(I \cap d^2\Z) + 2F \#(I\cap\Z).
	\end{align*}
	Here we used Lemma 5.3.4 in \cite{cohen2013cant} in the last step again.
	But since $\Delta/\tilde{f}^2$ is square--free we obtain
	\begin{align*}
			\cm(\Delta; \xi; \varepsilon)
			\le 8\varepsilon F \sum_{\twostack{d|\tilde{f}}{d \le |\Delta/3|^{1/4}}} d\left( \frac{|I|}{d^2} + 1\right)
					+ 2F\left(|I| + 1\right),
	\end{align*}
	where $|I|$ is the length of $I$. This can be further simplified to
	\begin{align*}
		\cm(\Delta; \xi; \varepsilon) 
			&\le 8\varepsilon F|I| \sum_{\twostack{d|\tilde{f}}{d \le |\Delta/3|^{1/4}}} d^{-1}
			+ 8\varepsilon F\sum_{\twostack{d|\tilde{f}}{d \le |\Delta/3|^{1/4}}} d
					+ 2F\left(|I| + 1\right)	\\
			&\le 8\varepsilon F|I| \frac{\sigma_1(\tilde f)}{\tilde f}
			+ 8\varepsilon F \left|\frac{\Delta}{3}\right|^{1/4} \sigma_0(\tilde f)
					+ 2F\left(|I| + 1\right).
	\end{align*}
	The length of $I$ can be estimated by
	\begin{displaymath}
		\frac{|\Delta|^{1/2}}{2y - 2\varepsilon} - \frac{|\Delta|^{1/2}}{2y + 2\varepsilon}
		= |\Delta|^{1/2}\frac{2y+2\varepsilon-(2y-2\varepsilon)}{4y^2-4\varepsilon^2}
		\le |\Delta|^{1/2}\frac{4\varepsilon}{4y^2-1}.
	\end{displaymath}
	This gives the desired inequality.
\end{proof}

The next corollary gives a bound on $\cm(\Delta;\xi;\varepsilon)$ just in terms of $\Delta$ and $\varepsilon$.

\begin{cor}\label{cor:eps_nbh}
	For $|\Delta| \ge 10^{14}$ and $0 < \varepsilon < 1/4$ we have
	\begin{displaymath}
		\cm(\Delta; \xi; \varepsilon)
		\le F(\Delta)\left(32 |\Delta|^{1/2} \varepsilon^2 \log\log(|\Delta|^{1/2})
				+ 11\vert\Delta\vert^{1/2}\varepsilon
				+ 2 \right).
	\end{displaymath}
\end{cor}

\begin{proof}
	For $|\Delta| \ge 10^{14}$ we can find the following results as Lemma 2.8 in  \cite{bhk}
	\begin{align*}
		\sigma_0(\tilde f) &\le |\Delta|^{0.192} \le |\Delta|^{1/4}	\\
		\sigma_1(\tilde f)/\tilde{f} &\le 1.842\log\log(|\Delta|^{1/2}).
	\end{align*}
	Moreover, we have $y \ge \sqrt{3}/2$ and thus $4y^2 - 1 \ge 2$. Hence
	\begin{displaymath}
		\frac{\sigma_1\left(\tilde f\right)}{\tilde f} \frac{|\Delta|^{1/2}}{4y^2 - 1} \varepsilon^2
		\le |\Delta|^{1/2} \varepsilon^2 \log\log(|\Delta|^{1/2})
	\end{displaymath}
	and
	\begin{align*}
		8 \sigma_0\left(\tilde f\right) \left\vert\frac{\Delta}3\right\vert^{1/4}\varepsilon
			+ 8\frac{|\Delta|^{1/2}}{4y^2 - 1} \varepsilon
		&\le \frac 8{3^{1/4}}\left\vert\Delta\right\vert^{1/2}\varepsilon
			+ 4|\Delta|^{1/2} \varepsilon	\\
		&\le 7\left\vert\Delta\right\vert^{1/2}\varepsilon
			+ 4|\Delta|^{1/2} \varepsilon,
	\end{align*}
	which gives the claimed statement.
\end{proof}

\section{Height bounds}

From now on $\alpha$ will be the $j$--invariant of an elliptic curve without complex multiplication.
As a preparation we will start with some properties of the $j$--function.

\begin{lem}\label{lem:mono}
	The function $j(1/2 +iy)$ as a function of $y$ on the interval $[\sqrt{3}/2, \infty)$ is real
	and decreasing.
	The function $j(e^{i\theta})$ on the interval $[\pi/3,\pi/2]$ is real and increasing,
	and we have $j(e^{i\pi/2}) = j(i) = 1728$.
	The function $j(iy)$ on the interval $[1, \infty)$ is real and increasing.
\end{lem}

\begin{proof}
	Recall $q = e^{2\pi i \tau}$. For $\tau = \frac 12 + iy$ with $y \ge \sqrt{3}/2$ we have
	$q = e^{\pi i} e^{-2\pi y} = -e^{-2\pi y}$. Thus $j(\tau)$ is real since all
	non--zero coefficients of $j$ are positive integers.
	See for example \cite{lehmer1942properties} for details.
	We have $j(1/2 + i\sqrt{3}/2) = 0$
	and from page 227 of \cite{cox2011primes} we know $j(1/2 + \sqrt{-7}/2) = -15^3$.
	But the map $y \mapsto j(1/2 + iy)$ is continuous and injective
	because $j$ is continuous and injective as a function
	on $\F$. Thus, it is monotonically decreasing.

	Similarly, if $\tau = iy$ with $y \ge 1$, then $q = e^{-2\pi y}$. We know $j(i) = 1728 = 12^3$
	and again from page 227 of \cite{cox2011primes} we know $j(i\sqrt{2}) = 20^3$.
	The same argument as before shows the claim for the map $y \mapsto j(iy)$.

	It remains to show that $j(e^{i\theta})$ is real because in that case $j(e^{i\pi/3}) = 0$
	and $j(e^{i\pi/2}) = 1728$ imply the monotonicity.
	Write $\tau = e^{i \theta}$.
	We have $\bar{q} = e^{2 \pi i (-\bar\tau)}$ and
	\begin{displaymath}
		\overline{j(\tau)} = (\bar{q})^{-1} + \sum_{n=0}^\infty c_n (\bar{q})^n = j(-\bar\tau).
	\end{displaymath}
	But $j$ is $\sl2z$--invariant so that $\overline{j(\tau)} = j(-\bar\tau) = j(\tau)$ since
	$|\tau| = 1$. Therefore, $j(\tau)$ must be real. This completes the proof.
\end{proof}
The next two statements tell us something about the growth of $j(\tau)$ as $|\tau|$ goes to infinity.
\begin{prop}\label{prop:jgrowth}
	If $\tau$ is in $\F$, then $\left||j(\tau)| - e^{2\pi \Im(\tau)}\right| \le 2079$.
\end{prop}
This result can be found in Lemma 1 of \cite{bilumasserzannier}.
We are going to prove the following result, which is of similar nature.

\begin{lem}\label{lem:boundj}
	Let $\tau$ be complex with $\Im(\tau) \ge 1/2$. Then
	\begin{displaymath}
		\left\vert \left\vert j(\tau)\right\vert - e^{2\pi\Im(\tau)} \right\vert \le 287473.
	\end{displaymath}
\end{lem}

\begin{proof}
	We have $j(\tau) = q^{-1} + c_0 + c_1 q + \cdots$, where as usual $q = e^{2\pi i \tau}$.
	Recall that the coefficients of the $q$--expansion of $j$ are all non--negative integers.
	Then
	$\left\vert|j|-|q^{-1}|\right\vert \le \sum_{n=0}^\infty c_n |q|^n \le \sum_{n=0}^\infty c_n q_0^n$ with
	$q_0 = e^{2\pi i \tau_0} = e^{-\pi}$ and $\tau_0 = i/2$. The right--hand side of the inequality
	is equal to $j(\tau_0) - q_0^{-1} = 66^3 - e^\pi \le 287473$.
	Note that we have used $j(\tau_0) = j(-1/\tau_0)$ since the $j$--function is $\sl2z$--invariant, and
	that $j(-1/\tau_0) = j(2i) = 66^3$ by Table 12.20 in \cite{cox2011primes}.
\end{proof}

In Lemma \ref{lem:mono} we proved that the $j$--function is real on the vertical and unit
circle geodesics of the fundamental domain and on the imaginary axis.
We can even say that the $j$--function is not real outside of this set, as the following statement shows.

\begin{cor}\label{cor:imag_j}
	If $\tau \in \F$ with $\Re(\tau) \not= 0, \pm \frac 12$ and $|\tau| > 1$, then $\Im(j(\tau)) \not= 0$.
	Moreover, $\Im(j(\tau)) < 0$ for $0 < \Re(\tau) < 1/2$ and $\Im(j(\tau)) > 0$ for $-1/2 < \Re(\tau) < 0$.
\end{cor}

\begin{proof}
	The proof is just an application of the intermediate value theorem.
	We use that $j$ is injective on $\F$. Assume $j(\tau) = R$ is real with $|\tau| > 1$ and $-1/2 < \Re(\tau) < 0$
	or $0 < \Re(\tau) < 1/2$.
	If $0 \le R \le 1728$, then $j(e^{i\theta}) = R$ for some $\pi/3 \le \theta \le \pi/2$ by the intermediate value
	theorem applied to the real function $t \mapsto j(e^{it})$. This is a contradiction to the injectivity
	of $j$ on $\F$ since $|\tau| > 1$.

	Assume $R \ge 1728$. By Lemma \ref{lem:mono} $j(iR) \ge 1728$ and applying Proposition \ref{prop:jgrowth}
	we have $j(iR) \ge e^{2\pi R} - 2079$. Thus $j(iR) \ge R$ and applying the intermediate value theorem again
	gives a $t \ge 1$ with $j(it) = R$. This is a contradiction since $0 < \tau < 1/2$.
	The case when $R<0$ follows similarly.

	To show $\Im(j(\tau)) < 0$ for $0 < \Re(\tau) < 1/2$ and $|\tau| > 1$ we assume we have
	$\tau_0, \tau_1$ in the interior of the fundamental domain $\F^\circ$ with positive real part
	and such that $\Im(j(\tau_0)) < 0$ and $\Im(j(\tau_1)) > 0$.
	Choose a path in $\F^\circ$, parametrized by $\gamma\colon [0,1] \rightarrow \F$, such that
	$\gamma(0) = \tau_0$, $\gamma(1) = \tau_1$ and such that every point in $\gamma([0,1])$
	is in the interior of $\F$ and has positive real part.
	The function $t \mapsto \Im(j(\gamma(t)))$ is continuous and satisfies
	$\Im(j(\gamma(0))) = \Im(j(\tau_0)) < 0$ and $\Im(j(\gamma(1))) = \Im(j(\tau_1)) > 0$.
	By the intermediate value theorem we have a $0 < t < 1$ with $\Im(j(\gamma(t))) = 0$
	which is impossible by the choice of $\gamma$ and the first claim of the corollary.
	So it suffices to give a value of $j(\tau)$ with $0 < \Re(\tau) < 1/2$ and $\Im(j(\tau)) < 0$.
	We have
	\begin{displaymath}
		j\left(\frac{1+5i}4\right) = -1728\left(\sqrt{5}-2\right)^{20}\left(3-2\sqrt{\sqrt{5}}i\right)^6
		\left(238\sqrt{5} + \frac{861}2 + 60\sqrt{\sqrt{5}}i\right)^3
	\end{displaymath}
	by page 17 of \cite{adlaj2014multiplication}.
	A computation with Sage shows $\Im\left(j\left(\frac{1+5i}4\right)\right) < 0$.
	We must have $\Im(j(\tau)) > 0$ for $-1/2 < \Re(\tau) < 0$ by the same argument and
	the fact that $j\colon \F \rightarrow \C$ is surjective.
	This completes the proof.
\end{proof}

The following two lemmas for $h(j)$ can be found in \cite{bhk}. The proofs follow directly from the statements in that
very paper with the inequality $h(j-\alpha) \ge h(j) - h(\alpha) - \log 2$.
For details see Proposition 4.1 and Proposition 4.3 in \cite{bhk}.
\begin{lem}\label{lem:lower_trivial}
	We have $[\Q(j):\Q] = \cm(\Delta)$, and if $|\Delta| \ge 16$, then
	\begin{displaymath}
		h(j-\alpha) \ge \frac{\pi|\Delta|^{1/2}-0.01}{\cm(\Delta)} - h(\alpha) - \log 2.
	\end{displaymath}
\end{lem}

\begin{proof}
	The first statement is a classical result, see for example Chapter 13 in \cite{cox2011primes}.
	The rest follows from Proposition 4.1 in \cite{bhk} and
	$h(j) = h(j-\alpha+\alpha) \le h(j-\alpha) + h(\alpha) + \log 2$.
\end{proof}

\begin{lem}\label{lem:lower_hard}
	We have
	\begin{displaymath}
		h(j-\alpha) \ge \frac 3{\sqrt 5} \log|\Delta| - 9.79 - h(\alpha) - \log 2.
	\end{displaymath}
\end{lem}

This lemma is more delicate and follows from work of Colmez \cite{colmez}
and Nakkajima--Taguchi \cite{nakkajimataguchi}.
The two previous lemmas bound the height of $j-\alpha$ from below.
Next we want to bound the height of $j-\alpha$ from above when $j-\alpha$ is an algebraic unit.

The following lemma says that if two points in the fundamental domain are close together, then the
difference of the images under the $j$--function can be bound from below in terms of the difference
of the points. Recall that $\zeta = e^{2\pi i/6}$.

\begin{lem}\label{lem:lin_log}
	Let $\zeta, \zeta^2 \not= \xi \in \bar\F$. Put $B = 4\cdot 10^5\max\{1,|j(\xi)|\}$
	and $A = |j''(i)|$ in the case when $\xi = i$ and $A = |j'(\xi)|$ otherwise.
	For $|\tau - \xi| \le \frac{A}{12A + 108B} \le \frac 13$
	we have
	\begin{displaymath}
		|j(\tau) - j(\xi)|
		\ge \frac{A}4 |\tau - \xi|^2.
	\end{displaymath}
	If $\xi \not= i$ we even have
	\begin{displaymath}
		|j(\tau) - j(\xi)| \ge \frac{A}2 |\tau - \xi|
	\end{displaymath}
	for $|\tau - \xi| \le \frac{A}{6A + 18B}$.
\end{lem}

Note that we can write $j''(i) = -2\cdot 3^4 \Gamma(1/4)^8/\pi^4$ as
K\"uhne shows in the appendix of \cite{wustholz2014note}.
Since $\Gamma(1/4) = 3.6256\ldots > 3$ we could further estimate the first of the
two bounds in the lemma by
\begin{displaymath}
	|j(\tau) - 1728| \ge 12413|\tau - i|^2.
\end{displaymath}

\begin{proof}
	This is a special case of Lemma 2.4 in \cite{bilulucapizarro}. We take $f(\tau) = j(\tau) - j(\xi)$.
	Assume $|\tau - \xi| \le 1/3$. Then by Lemma \ref{lem:boundj}
	\begin{align*}
		|f(\tau)|  &\le |j(\tau)| + |j(\xi)| \le |j(\xi)| + e^{2\pi \Im(\tau)} + 287473	\\
		&\le |j(\xi)| + e^{2\pi(\Im(\xi) + 1/3)} + 287473.
	\end{align*}
	We have $e^{2\pi/3} < 9$ so applying Proposition \ref{prop:jgrowth} we obtain
	\begin{align*}
		|f(\tau)| &\le |j(\xi)| + e^{2\pi(\Im(\xi) + 1/3)} + 287473	\\
		&\le |j(\xi)| + 9|j(\xi)| + 9\cdot 2079 + 287473	\\
		&\le 10|j(\xi)| + 306184	\\
		&\le 4\cdot 10^5\max\{1,|j(\xi)|\}.
	\end{align*}
	We treat the two cases $\xi = i$ and $\xi \not= i$ separately and start with the latter.
	By \cite{bilulucapizarro} we have
	\begin{displaymath}
		|j(\tau) - j(\xi)| - A|\tau - \xi| \ge -\frac{A/3 + B}{(1/3)^2} |\tau - \xi|^2
		= -\left(3A + 9B\right) |\tau - \xi|^2,
	\end{displaymath}
	where $A = |j'(\xi)|$.
	In the smaller disc $|\tau - \xi| \le \frac{A}{2\cdot(3A + 9B)}$ we then obtain
	\begin{displaymath}
		|j(\tau) - j(\xi)| \ge \frac{A}{2} |\tau - \xi|.
	\end{displaymath}
	Now assume $\xi = i$ and put $A = |j''(i)|$. By \cite{bilulucapizarro} we conclude
	\begin{displaymath}
		|j(\tau) - j(\xi)| - \frac A2|\tau - \xi|^2 \ge -\frac{A/9 + B}{(1/3)^3} |\tau - \xi|^3
		= -\left(3A + 27B\right) |\tau - \xi|^3
	\end{displaymath}
	and thus
	\begin{displaymath}
		|j(\tau) - j(\xi)| \ge \frac{A}{4} |\tau - \xi|^2.
	\end{displaymath}
	in the smaller disc $|\tau - \xi| \le \frac{A}{4\cdot(3A + 27B)}$.
\end{proof}

We can also bound the $j$--invariants outside of a neighborhood of $\xi$ as the following lemma shows.
We write
\begin{displaymath}
	\F_+ = \{ \tau \in \F; 0 \le \Re(\tau) \le 1/2 \}
\end{displaymath}
and
\begin{displaymath}
	\F_- = \{ \tau \in \F; -1/2 \le \Re(\tau) \le 0 \}.
\end{displaymath}

\begin{lem}\label{lem:jboundbelow}
	Let $\zeta \not=\xi,\tau \in \F_+$ and put $A = |j''(i)|$ if $\xi = i$ and $A = |j'(\xi)|$ otherwise.
	Also define $B = 4\cdot10^5 \max\{1,|j(\xi)|\}$.
	Assume $|\tau - \xi| \ge \delta$,
	where $\delta$ is defined as the minimum of $\frac{A}{12A + 108B}$ and half the (euclidean)
	distance of $\xi$ to any geodesic of $\partial\F_+$ (i.e.~the vertical line
	segments with real part $0$ and $1/2$ and the part of the unit circle
	between those lines) not containing $\xi$.
	Then
	\begin{displaymath}
		|j(\tau) - j(\xi)| \ge c(\xi)
	\end{displaymath}
	where $c(\xi) > 0$ is a constant depending on $\xi$ and is given by
	\begin{displaymath}
		c(\xi) = \begin{cases}
			A\delta/2	&	\text{if } \xi \in \partial\F_+\setminus\{i\}	\\
			A\delta^2/4	&	\text{if } \xi = i	\\
			\min\left\{|\Im(j(\xi))|, A\delta/2 \right\}	&	\text{otherwise.} 	\\
		\end{cases}
	\end{displaymath}
\end{lem}

\begin{proof}
	We want to apply the maximum modulus principle. To do this we will give lower bounds on the boundary
	\begin{figure}
		\begin{center}
			\pgfmathsetmacro{\myxlow}{-1}
			\pgfmathsetmacro{\myxhigh}{1}
			\pgfmathsetmacro{\myiterations}{2}
			\begin{tikzpicture}[scale=3]
				\draw[-latex](0,1) -- (0,2.6);

				\foreach \y  in {1,1.2,...,2}
				{   \draw[blue] (1/2,\y) -- (0,{\y+0.2}) {};
				}
				\draw[blue] (1/2,2.2) -- (1/4,2.3) {};

				\coordinate (zeta) at (1/2,{sqrt(3)/2});
				\coordinate (zeta2) at (-1/2,{sqrt(3)/2});

				\draw[thick] (zeta) -- (1/2,2.5);
				\draw[blue,thick,domain=60:90] plot ({cos(\x)}, {sin(\x)});
				\draw[thick,domain=90:120] plot ({cos(\x)}, {sin(\x)});
				\draw[blue,thick] (zeta) -- (1/2,2.3);
				\draw[blue,thick] (0,2.3) -- (1/2,2.3);
				\draw[blue,thick] (0,1) -- (0,2.3);
				\draw[thick] (zeta2) -- (-1/2,2.5);

				\pgfmathsetmacro{\eps}{.15}

				\coordinate (xi) at (1/4, 1.3);
				\draw[fill=white] (xi) circle (\eps) {};
				\draw[fill=black] (xi) circle (.3pt) node[right,font=\tiny] {$\xi$};
				\draw[thick,blue] (xi) circle (\eps);
				\draw[->] (xi) -- ++(115:\eps) node[below right,font=\tiny] {$\delta$};
			\end{tikzpicture}
		\end{center}
		\caption{\small Application of the maximum modulus principle to the blue area.}\label{fig:mmp}
	\end{figure}
	of $\F_+$, see Figure \ref{fig:mmp}.
	By the previous lemma we have
	\begin{displaymath}
		|j(\tau) - j(\xi)| \ge \frac{A}2 \delta
	\end{displaymath}
	or
	\begin{displaymath}
		|j(\tau) - j(i)| \ge \frac{A}4 \delta^2
	\end{displaymath}
	on the circle $|\tau - \xi| = \delta$.
	Also $\Im(\tau) \ge 6\Im(\xi)$ implies $\Im(\tau) \ge \Im(\xi) + 1$,
	so we obtain by applying Proposition \ref{prop:jgrowth} twice
	\begin{align*}
		|j(\xi)| &\le 2079 + e^{2\pi \Im(\xi)}
		= 2079 - 20e^{2\pi\frac{\sqrt{3}}{2}} + 20e^{2\pi\frac{\sqrt{3}}{2}} + e^{2\pi \Im(\xi)}	\\
		&< -2079 + 20e^{2\pi\frac{\sqrt{3}}{2}} + e^{2\pi \Im(\xi)}
		\le -2079 + 20e^{2\pi\Im(\xi)} + e^{2\pi \Im(\xi)}	\\
		&\le -2079 + 21e^{2\pi\Im(\xi)}
		< -2079 + e^{2\pi}e^{2\pi\Im(\xi)}	\\
		&\le -2079 + e^{2\pi\Im(\tau)}	\\
		&\le |j(\tau)|.
	\end{align*}
	Thus, by using Proposition \ref{prop:jgrowth} twice
	we get
	\begin{align*}
		|j(\tau) - j(\xi)| &\ge |j(\tau)| - |j(\xi)|	\\
		&\ge  -2079 + e^{2\pi \Im(\tau)} - (2079 + e^{2\pi \Im(\xi)})	\\
		&\ge  -2\cdot 2079 + e^{2\pi \Im(\tau)} - e^{2\pi \Im(\xi)}.
	\end{align*}
	Now we use $\Im(\tau) \ge 6\Im(\xi)$ to get
	\begin{displaymath}
		|j(\tau) - j(\xi)| \ge  -2\cdot 2079 + e^{12\pi \Im(\xi)} - e^{2\pi \Im(\xi)}.
	\end{displaymath}
	We have $e^{2x} \ge 2e^x$ for any $x \ge \log 2$ and $2\pi \Im(\xi) \ge 1$
	so that
	\begin{align*}
		|j(\tau) - j(\xi)|
		&\ge  -2\cdot 2079 + e^{4\pi \Im(\xi)} + e^{2\pi \Im(\xi)} + e^{4\pi \Im(\xi)} - e^{2\pi \Im(\xi)}	\\
		&=  -2\cdot 2079 + e^{4\pi \Im(\xi)} + e^{4\pi \Im(\xi)}.
	\end{align*}
	But also because of $\Im(\xi) \ge \sqrt{3}/2$ we get
	\begin{equation}\label{eq:j_cut_above}
		|j(\tau) - j(\xi)| \ge e^{4\pi \Im(\xi)}.
	\end{equation}
	We have $\delta \le 1/12$ by definition and Lemme 1 of \cite{faisant1987quelques} gives for $\xi \not= i$
	\begin{align*}
		\frac A2 \delta &\le \frac{A}{24} \le \frac{8\pi}{24}e^{2(\pi +1)\max\{\Im(\xi),\Im(\xi)^{-1}\}}	\\
		&\le \frac{8\pi}{24}e^{3\pi \max\{\Im(\xi),\Im(\xi)^{-1}\}}.
	\end{align*}
	So if $\Im(\xi) \ge 1$, then
	\begin{displaymath}
		|j(\tau) - j(\xi)| \ge \frac A2 \delta.
	\end{displaymath}
	If $\sqrt{3}/2 \le \Im(\xi) \le 1$, then
	\begin{align*}
		\frac A2 &\delta \le \frac{8\pi}{24}e^{3\pi \Im(\xi)^{-1}}
		\le \frac{8\pi}{24}e^{2\pi/\sqrt{3}}
		\le e^{4 \pi \cdot \sqrt{3}/2}
		\le e^{4 \pi \Im(\xi)}.
	\end{align*}
	So in any case
	\begin{displaymath}
		|j(\tau) - j(\xi)| \ge \frac A2\delta
	\end{displaymath}
	for $\xi \not= i$.
	If $\xi = i$, then as described after Lemma \ref{lem:lin_log} we get
	\begin{displaymath}
		\frac A4\delta^2 \le 12414\delta^2 \le 87
	\end{displaymath}
	which together with \eqref{eq:j_cut_above} gives
	\begin{displaymath}
		|j(\tau) - j(i)| \ge \frac A4\delta^2.
	\end{displaymath}

	So we have treated all the $\tau$ with large imaginary part.
	Now we have to go through the different cases for $\xi$ to bound the boundary.
	We start with $\Re(\xi) = 1/2$ so that we are in the following case.
		\begin{center}
			\pgfmathsetmacro{\myxlow}{-1}
			\pgfmathsetmacro{\myxhigh}{1}
			\pgfmathsetmacro{\myiterations}{2}
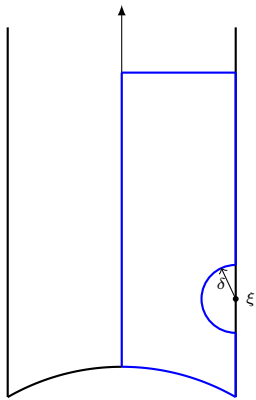
			\begin{tikzpicture}[scale=3]
				\draw[-latex](0,1) -- (0,2.6);

				\coordinate (zeta) at (1/2,{sqrt(3)/2});
				\coordinate (zeta2) at (-1/2,{sqrt(3)/2});

				\pgfmathsetmacro{\eps}{.15}

				\coordinate (xi) at (1/2, 1.3);

				\draw[thick] (zeta) -- (1/2,2.5);
				\draw[blue,thick,domain=60:90] plot ({cos(\x)}, {sin(\x)});
				\draw[thick,domain=90:120] plot ({cos(\x)}, {sin(\x)});
				\draw[blue,thick] (zeta) -- ($(xi)-(0,\eps)$);
				\draw[blue,thick] ($(xi)+(0,\eps)$) -- (1/2,2.3);
				\draw[blue,thick] (0,2.3) -- (1/2,2.3);
				\draw[blue,thick] (0,1) -- (0,2.3);
				\draw[thick] (zeta2) -- (-1/2,2.5);

				\draw[fill=black] (xi) circle (.3pt) node[right,font=\tiny] {$\xi$};
				\draw[thick,blue] (xi) ++(90:\eps) arc (90:270:\eps);
				\draw[->] (xi) -- ++(115:\eps) node[below,font=\tiny] {$\delta$};
			\end{tikzpicture}
		\end{center}
		For the remainder of the proof we will be using that $j$ is monotonically increasing or decreasing
		on the boundary as shown in Lemma \ref{lem:mono}.
		If $\tau$ is on the same boundary component as $\xi$, then $\Im(\tau) \ge \Im(\xi) + \delta$
		or $\Im(\tau) \le \Im(\xi) - \delta$. Therefore, by monotonicity either $|j(\tau)| > |j(\xi)|$ which
		implies
		\begin{displaymath}
			|j(\tau) - j(\xi)| \ge j(\xi) - j(\tau) \ge j(\xi) - j(\xi + i\delta) = |j(\xi) - j(\xi + i\delta)| \ge \frac A2\delta
		\end{displaymath}
		or $|j(\tau)| < |j(\xi)|$ and
		\begin{displaymath}
			|j(\tau) - j(\xi)| \ge |j(\xi)| - |j(\tau)| \ge j(\xi - i\delta) - j(\xi) = |j(\xi) - j(\xi - i\delta)| \ge \frac A2\delta.
		\end{displaymath}
		Note that the last inequality of both displays follows from Lemma \ref{lem:lin_log} on the boundary
		$|\tau-\xi| = \delta$.
		If $|\tau| = 1$ or $\Re(\tau) = 0$, then $j(\tau) \ge 0$ and $j(\xi) < 0$, so that
		\begin{align*}
			|j(\tau) - j(\xi)| &= j(\tau) - j(\xi) \ge -j(\xi)	\\
			&\ge j(\xi-i\delta) - j(\xi) = |j(\xi) - j(\xi-i\delta)|	\\
			&\ge \frac A2 \delta.
		\end{align*}
		Here we have used Lemma \ref{lem:lin_log} for the last estimate.
		Altogether, if $\Re(\xi) = 1/2$ we get by the minimum modulus principle
		\begin{displaymath}
			|j(\tau) - j(\xi)| \ge \frac A2 \delta
		\end{displaymath}
		as desired.

		The second case is when $\Re(\xi) = 0$ and $\xi \not= i$.
		We have $j(\tau) \le 0$ if $\Re(\tau) = 1/2$, and thus
		\begin{align*}
			|j(\xi) - j(\tau)| &\ge j(\xi) \ge j(\xi) - j(\xi-i\delta)	\\
			&= |j(\xi) - j(\xi-i\delta)|	\\
			&\ge \frac A2 \delta.
		\end{align*}
		If $|\tau| = 1$, then again by monotonicity and Lemma \ref{lem:lin_log}
		\begin{align*}
			|j(\xi) - j(\tau)| &\ge j(\xi) - j(\tau) \ge j(\xi) - j(\xi-i\delta)	\\
			&= |j(\xi)-j(\xi-i\delta)|	\\
			&\ge \frac A2 \delta.
		\end{align*}
		For the case when $\Re(\tau) = 0$ we have two subcases. When $\Im(\tau) \ge \Im(\xi) + \delta$,
		it follows $j(\tau) > j(\xi)$ and hence
		\begin{equation}\label{eq:iabove}
			|j(\xi) - j(\tau)| \ge j(\xi + i\delta) - j(\xi) \ge \frac A2\delta.
		\end{equation}
		Or we have $\Im(\tau) \le \Im(\xi) - \delta$ and we get
		\begin{displaymath}
			|j(\xi) - j(\tau)| \ge j(\xi) - j(\xi - i\delta) \ge \frac A2\delta.
		\end{displaymath}
		In sum, by applying the minimum modulus principle we obtain
		\begin{equation}\label{eq:boundary_bound}
			|j(\tau) - j(\xi)| \ge \frac A2 \delta.
		\end{equation}

		The third case is $|\xi| = 1$ and $\xi \not= i$.
		Write $\xi = e^{i\theta}$.
		Again we have three subcases. If $\Re(\tau) = 1/2$, then $j(\tau) \le 0$ and
		\begin{equation}\label{eq:circle2right}
			\begin{aligned}
				|j(\xi) - j(\tau)| &\ge j(\xi) - j(\tau) \ge j(\xi)	\\
				&\ge j(\xi) - j\left(e^{i(\theta - 2 \arcsin(\delta/2))}\right)	\\
				&= |j(\xi)-j\left(e^{i(\theta - 2 \arcsin(\delta/2))}\right)|	\\
				&\ge \frac A2\delta.
			\end{aligned}
		\end{equation}
		Note that $e^{i(\theta - 2 \arcsin(\delta/2))}$ is one of the two points where the
		circle of radius $\delta$ and the unit circle intersect.
		If $\Re(\tau) = 0$, then $j(\tau) \ge 1728 > j\left(e^{i(\theta + 2 \arcsin(\delta/2))}\right) > j(\xi)$ and
		\begin{align*}
			|j(\xi) - j(\tau)| &\ge j(\tau) - j(\xi) \ge j\left(e^{i(\theta + 2 \arcsin(\delta/2))}\right) - j(\xi)	\\
			&= |j(\xi) - j\left(e^{i(\theta + 2 \arcsin(\delta/2))}\right)|	\\
			&\ge \frac A2\delta.
		\end{align*}
		If $|\tau| = 1$, then $j(\tau) < j(\xi)$ or $j(\tau) > j(\xi)$ and Lemma \ref{lem:lin_log}
		tells us
		\begin{equation}\label{eq:itotheright}
			|j(\xi) - j(\tau)| \ge \left|j(\xi) - j\left(e^{i(\theta \pm 2 \arcsin(\delta/2))}\right)\right| \ge \frac A2\delta.
		\end{equation}
		By applying the minimum modulus principle we obtain the same result as in equation \eqref{eq:boundary_bound}.

		If $\xi = i$, then the estimates in equations \eqref{eq:iabove} and \eqref{eq:itotheright} hold
		with $A\delta/2$ replaced by $A\delta^2/4$. Also equation \eqref{eq:circle2right} holds with
		$\xi$ replaced by $i$ and $A\delta/2$ replaced by $A\delta^2/4$, and thus by the minimum modulus
		principle
		\begin{displaymath}
			|j(\tau) - j(i)| \ge \frac A4 \delta^2.
		\end{displaymath}

		The last case is $0 < \Re(\xi) < 1/2$ and $|\xi| > 1$. Let $\tau \in \partial\F_+$. Then
		$j(\tau)$ is real and we have
		\begin{displaymath}
			|j(\xi) - j(\tau)| \ge |\Im(j(\xi)) - \Im(j(\tau))| = |\Im(j(\xi))|.
		\end{displaymath}
		This is the case shown in Figure \ref{fig:mmp}.
		Applying the minimum modulus principle gives the desired result.
\end{proof}

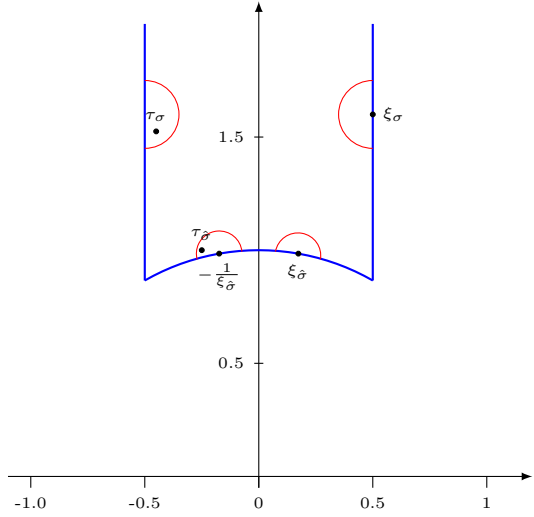
\begin{figure}
	\begin{center}
		\pgfmathsetmacro{\myxlow}{-1}
		\pgfmathsetmacro{\myxhigh}{1}
		\pgfmathsetmacro{\myiterations}{2}
		\begin{tikzpicture}[scale=3]
			\draw[-latex](\myxlow-0.1,0) -- (\myxhigh+0.2,0);
			\pgfmathsetmacro{\succofmyxlow}{\myxlow+0.5}
			\foreach \x in {\myxlow,\succofmyxlow,...,\myxhigh}
			{   \draw (\x,0) -- (\x,-0.05) node[below,font=\tiny] {\x};
			}
			\foreach \y  in {0.5,...,1.5}
			{   \draw (0.02,\y) -- (-0.02,\y) node[left,font=\tiny] {\pgfmathprintnumber{\y}};
			}
			\draw[-latex](0,-0.05) -- (0,2.1);

			\coordinate (zeta) at (1/2,{sqrt(3)/2});
			\coordinate (zeta2) at (-1/2,{sqrt(3)/2});

			\draw[blue,thick,domain=60:120] plot ({cos(\x)}, {sin(\x)});
			\draw[blue,thick] (zeta) -- (1/2,2);
			\draw[blue,thick] (zeta2) -- (-1/2,2);

			\pgfmathsetmacro{\eps}{.15}

			\coordinate (xi) at (1/2, 1.6);
			\coordinate (xi_1) at (-1/2, 1.6);
			\coordinate (tau) at ({-1/2+\eps/3}, {1.6-\eps/2});
			\draw[fill=black] (xi) circle (.3pt) node[right,font=\tiny] {$\xi_\sigma$};
			\draw[red] (xi) ++(90:\eps) arc (90:270:\eps);
			\draw[red] (xi_1) ++(-90:\eps) arc (-90:90:\eps);
			\draw[fill=black] (tau) circle (.3pt) node[above,font=\tiny] {$\tau_\sigma$};

			\pgfmathsetmacro{\eps}{.1}
			\coordinate (xi) at ({cos(80)},{sin(80)});
			\coordinate (xi1) at ({cos(100)},{sin(100)});
			\coordinate (tau) at (-0.25, 1);
			\draw[fill=black] (xi) circle (.3pt) node[below,font=\tiny] {$\xi_{\hat \sigma}$};
			\draw[fill=black] (xi1) circle (.3pt) node[below,font=\tiny] {$-\frac 1{\xi_{\hat \sigma}}$};
			\draw[red] (xi) ++(-15:\eps) arc (-10:170:\eps);
			\draw[red] (xi1) ++(5:\eps) arc (5:195:\eps);
			\draw[fill=black] (tau) circle (.3pt) node[above,font=\tiny] {$\tau_{\hat \sigma}$};
		\end{tikzpicture}
	\end{center}
	\caption{Neighborhoods of points on the boundary}\label{fig:nbh_boundary}
\end{figure}

Note that the same claim holds for $\F_-$ since we have the symmetry $\Re(j(x+iy)) = \Re(j(-x+iy))$
and $\Im(j(x+iy)) = -\Im(j(-x+iy))$ which directly follows from the $q$--expansion.
If $j(\tau)$ and $j(\xi)$ are close we want $|\tau_\sigma - \xi_\sigma|$ to be small too,
but we can not get this
in general as Figure \ref{fig:nbh_boundary} shows.

\begin{lem}\label{lem:st_jbound}
	In the same setting as in the previous lemma,
	if $|j(\tau) - j(\xi)| < c(\xi)$, then $|\tau - M\xi| < \delta$ with $M\xi \in \bar\F$
	for some $M \in \T$ where
	$\T = \{ \begin{smatrix}1&0\\0&1\end{smatrix}, \begin{smatrix}1&\pm 1\\0&1\end{smatrix}, \begin{smatrix}0&-1\\1&0\end{smatrix} \}$.
\end{lem}

\begin{proof}
	If $\xi, \tau$ are both in $\F_+$ or both in $\F_-$, then $|\tau - \xi| \le \delta$
	by Lemma \ref{lem:jboundbelow}.
	We now can assume without loss of generality $\Re(\xi) < 0$ and $\Re(\tau) > 0$.
	If $\xi$ is on the boundary of $\F$, then $\Re(\xi) = -1/2$.
	Then we can apply Lemma \ref{lem:jboundbelow} to $\tau$ and $\begin{smatrix}1&1\\0&1\end{smatrix}\xi$
	and use $j(\xi) = j(\xi+1)$.
	If $|\xi| = 1$ and $\Re(\xi) < 0$, then we can again apply Lemma \ref{lem:jboundbelow}
	to $\tau$ and $\begin{smatrix}0&-1\\1&0\end{smatrix}\xi$
	and use $j(\xi) = j\left(\begin{smatrix}0&-1\\1&0\end{smatrix}\xi\right)$.
	If $-1/2 < \Re(\xi) < 0$ and $\tau \in \F_+$, then by Corollary \ref{cor:imag_j}
	\begin{align*}
		|j(\xi) - j(\tau)| &\ge |\Im(j(\xi)) - \Im(j(\tau))|
		= \Im(j(\xi)) - \Im(j(\tau))	\\
		&= \Im(j(\xi)) + |\Im(j(\tau))|	\\
		&\ge \Im(j(\xi))	\\
		&\ge c(\xi)
	\end{align*}
	contrary to the assumption.
\end{proof}

The following lemma can be found in \cite{habSM} as Lemma 5.

\begin{lem}\label{lem:ht_less_D}
	Let $\sigma\colon \bar\Q \hookrightarrow \C$ and let $\tau$ be imaginary quadratic.
	Let $\tau_\sigma$ satisfy $j(\tau_\sigma) = \sigma(j(\tau))$ with $\tau_\sigma \in \F$.
	If $\Delta$ is the discriminant of the endomorphism ring associated to $j(\tau)$,
	then $\tau_\sigma$ is imaginary quadratic and $h(\tau_\sigma) \le \log \sqrt{|\Delta|}$.
\end{lem}

With these lemmas we are now able to bound the height from above using linear forms in logarithms.
Instead of using a result by Masser \cite{masserelliptic} we are going to use a result by
David \cite{davidlinearforms} to make the constants explicit.
We are going to use this lemma to prove a result on linear forms. It is a special case of
Th\'eor\`eme 2.1 in \cite{davidlinearforms}. The height of a matrix in $\sl2z$ will be
the height when regarded as a member of $\Q^4$.

\begin{lem}\label{lem:lin_forms}
	Let $\tau \in \H$ be imaginary quadratic and let $E\colon y^2 = 4x^3 - g_2 x - g_3$
	be an elliptic curve without complex multiplication such that $E(\C) \simeq \C/(\omega_1\Z+\omega_2\Z)$
	with $\xi = \omega_2/\omega_1 \in \F$.
	Assume $E$ is defined over a number field of degree $D$.
	Put $h = \max\{1, h(1,g_2,g_3), h(j(\xi))\}$ and assume $\Delta$ is the discriminant
	of the endomorphism ring of an elliptic curve with $j$--invariant $j(\tau)$.
	Pick any $M = \begin{smatrix}\alpha&\beta\\\gamma&\delta\end{smatrix}\in \sl2z$.
	If $|\Delta| \ge \max\{2D, e^{12\pi h}, 4H(M)^4\}$,
	then we have
	\begin{displaymath}
		\log|(\gamma\tau-\alpha)\omega_2 - (\delta\tau-\beta)\omega_1| \ge -c_2 (\log|\Delta|)^4
	\end{displaymath}
	with
	\begin{displaymath}
		c_2 = 2^{50} 3^{43} 5^{18} \cdot D^6 \cdot h^2.
	\end{displaymath}
\end{lem}

\begin{proof}
	We follow the notation in \cite{davidlinearforms}.
	Put $\mathcal{L}(z_0,z_1,z_2) = (\gamma\tau-\alpha)z_1 - (\delta\tau-\beta)z_2$.
	We choose the variables $u_1, u_2$ in the theorem to be $\omega_2$ and $\omega_1$, respectively.
	Then $\gamma_1 = \gamma_2 = (0,0,1)$ and $v = (1,\omega_2,\omega_1)$.
	Note that $\mathcal{L}(v) \not= 0$ since otherwise we would have $\tau = M.(\omega_2/\omega_1) = M\xi$
	but $\xi$ and thus $M\xi$ are not imaginary quadratic.
	We have to estimate the height of the coefficients of the linear form.
	By Lemma \ref{lem:ht_less_D} we have $h(\tau) \le \log\sqrt{|\Delta|}$.
	Thus the multiplicative height of the coefficients can be estimated by
	\begin{align*}
		H(\gamma\tau-\beta) &\le 2H(\gamma\tau)H(\beta) \le 2H(\gamma)H(\tau)H(\beta) \le 2H(M)^2H(\tau)	\\
		&\le 2H(M)^2|\Delta|^{1/2} \le |\Delta|
	\end{align*}
	and similarly
	\begin{align*}
		H(\delta\tau-\beta) \le |\Delta|.
	\end{align*}
	We thus can work with $B = |\Delta|$ which satisfies condition (2) of Th\'eor\`eme 2.1.
	Moreover, we can pick $V = V_1 = V_2 = e^{12\pi h}$. To see this note
	\begin{align*}
		\frac{3\pi |u_1|^2}{|\omega_1|^2 \Im(\xi) 2D}
		\le \frac{3\pi |u_1|^2}{|\omega_1|^2 \Im(\xi) D}
		\le \frac{3\pi |\xi|^2}{\Im(\xi) D}
		\le 4\pi \frac{|\xi|}D
	\end{align*}
	and also
	\begin{align*}
		\frac{3\pi |u_2|^2}{|\omega_1|^2 \Im(\xi) D}
		\le \frac{3\pi}{\Im(\xi) D}
		\le 4\pi.
	\end{align*}
	By Lemme 1 item (iv) in \cite{faisant1987quelques} we have
	\begin{align*}
		|\xi| \le \frac 32 \log \max\{e,|j(\xi)|\}
	\end{align*}
	and therefore
	\begin{align*}
		\frac{|\xi|}D &\le \frac 32 \frac 1D \log \max\{e,|j(\xi)|\}	\\
		&\le \frac 32 \frac 1D \left(1 + \log \max\{1,|j(\xi)|\}\right)	\\
		&\le \frac 32 \frac 1D \left(1 + \sum_{\nu \in M_K} d_\nu \log \max\{1,|j(\xi)|_\nu\}\right)	\\
		&\le \frac 32 + \frac 32 h(j(\xi)).
	\end{align*}
	Here $K$ denotes the field of definition for $E$.
	Thus conditions (1) and (3) of the theorem in \cite{davidlinearforms} are satisfied.
	We thus can apply the theorem and get the lower bound
	\begin{align*}
		\log \lvert \mathcal{L}(v) \rvert &\ge -C \cdot 2^6 D^6 (\log B + \log(2D) + 1)
														\cdot(\log\log B+h+\log(2D)+1)^3\log V_1 \log V_2	\\
					&\ge - 2^{16} 3^3 C\cdot D^6 \cdot h^2 \cdot (\log B)^4
	\end{align*}
	since $h \le \log B$ and $\log(2D) \le \log B$. The constant $C$ comes from \cite{davidlinearforms}
	and is given by
	\begin{displaymath}
		C = 10^{18}\cdot 4^8\cdot 3^{40}.
	\end{displaymath}
	This gives the desired inequality.
\end{proof}

For the following, if $\tau \in \F$ with $j(\tau)$ algebraic
and a field embedding $\sigma\colon \Q(\alpha) \hookrightarrow \C$, $\alpha \in \bar\Q$, are given,
then $\tau_\sigma \in \F$ is defined by $\sigma(j(\tau)) = j(\tau_\sigma)$.
Note that for fixed $\xi \in \F$ different from $\zeta, \zeta^2, i$ we have $j'(\xi_\sigma) \not= 0$.
Suppose that $j(\xi) = \alpha$ is algebraic.
We define the function
\begin{equation}\label{eq:penalty}
	\Pen(\xi) = \log\max_\sigma\left\{1,c(\xi_\sigma)^{-1}\right\},
\end{equation}
where $\sigma$ runs over all embeddings $\sigma\colon \Q(\alpha) \hookrightarrow \C$,
and $c(\xi_\sigma)$ is defined as in Lemma \ref{lem:jboundbelow}
Note that the expression in the maximum is larger than 12 since $\delta_\sigma \le 1/12$, and hence $\Pen(\xi) > 0$.
The function is large when some $\xi_\sigma$ is close to one of the three points $i, \zeta, \zeta^2$ or
$\xi_\sigma$ is close to the boundary of $\F$ or to the vertical imaginary axis.

\begin{prop}\label{prop:ht_bound_w_sum}
	Assume $j$ is a singular modulus.
	Let $\alpha = j(\xi)$, $\xi \in \F$, be an algebraic number that is the $j$--invariant of an elliptic curve
	without complex multiplication.
	Let $E\colon y^2 = 4x^3 - g_2 x - g_3$ be a model for $\alpha$ such that $E(\C) \simeq \C/(\omega_1\Z+\omega_2\Z)$
	and $\xi = \omega_2/\omega_1$. Similarly, choose a model for $E^\sigma$ for any $\sigma\colon \Q(\alpha) \hookrightarrow \C$.
	Further assume that $E$ is defined over a field of degree $D$ and $|\Delta| \ge \max\{2D, e^{12\pi h}\}$.
	Assume that $j-\alpha$ is an algebraic unit
	and let $0 < \varepsilon < 1/4$.
	We can bound the height by
	\begin{align*}
		h(j-\alpha)
		\le& c_2\frac{\sum_{\sigma\colon \Q(j,\alpha) \rightarrow \C}\sum_{M \in \T}\cm(\Delta;M\xi_\sigma;\varepsilon)}{16\cdot\cm(\Delta)}
				(\log|\Delta|)^4 + 5\Pen(\xi) + 4\mathcal{M}(\xi) +  |\log \varepsilon|,
	\end{align*}
	where $\T = \{ \begin{smatrix}1&0\\0&1\end{smatrix}, \begin{smatrix}1&\pm 1\\0&1\end{smatrix}, \begin{smatrix}0&-1\\1&0\end{smatrix} \}$
	and $\mathcal{M}(\xi) = \log\max_\sigma\{1,|\omega_{\sigma,1}|,|\omega_{\sigma,2}|\}$.
\end{prop}

\begin{proof}
	Let $j(\xi) = \alpha$.
	Since $j-\alpha$ is an algebraic unit the height can be computed by
	\begin{equation}\label{eq:height_sum}
		h(j-\alpha) = \frac 1{[\Q(j,\alpha):\Q]} \sum_{\sigma} \log \max\left\{1, |\sigma(j-\alpha)^{-1}|\right\},
	\end{equation}
	where $\sigma$ runs over all field embeddings $\sigma\colon \Q(j,\alpha) \rightarrow \C$.
	Let
	\begin{displaymath}
		\varepsilon_0 = \varepsilon \cdot \min_{\sigma\colon \Q(\alpha) \rightarrow \C}
				\left\{1, c(\xi_\sigma)\right\}.
	\end{displaymath}
	We split the sum \eqref{eq:height_sum} into terms for which $|\sigma(j-\alpha)| < \varepsilon_0$
	and those for which $1 \ge |\sigma(j-\alpha)| \ge \varepsilon_0$.

	Assume $|\sigma(j-\alpha)| < \varepsilon_0$.
	Thus $|\sigma(j-\alpha)| < c(\xi_\sigma)$.
	We want to show $|\tau_\sigma - M\xi_\sigma| < \varepsilon$ for some $M \in \T$.
	We can apply Lemma \ref{lem:st_jbound} to get $|\tau_\sigma - M\xi_\sigma| < \delta_\sigma$ for some $M \in \T$.
	By definition of $\delta_\sigma$ we have
 	$\delta_\sigma \le \frac{|j'(\xi_\sigma)|}{6|j'(\xi_\sigma)| + 72\cdot 10^5 \max\{1,|j(\xi_\sigma)|\}}$, so
	that we can apply Lemma \ref{lem:lin_log} and $j(\xi_\sigma) = j(M\xi_\sigma)$ to get
	\begin{displaymath}
		\frac{|j'(\xi_\sigma)|}2 |\tau_\sigma - M\xi_\sigma| \le |j(\tau_\sigma) - j(\xi_\sigma)| < \varepsilon_0
		\le \varepsilon c(\xi_\sigma) \le \varepsilon \frac{|j'(\xi_\sigma)|}2\delta_\sigma \le \varepsilon \frac{|j'(\xi_\sigma)|}2,
	\end{displaymath}
	which implies $|\tau_\sigma - M\xi_\sigma| < \varepsilon$.
	But we also get
	\begin{equation}\label{eq:j_to_lin_form}
		|j(\tau_\sigma) - j(\xi_\sigma)| \ge \frac{|j'(\xi_\sigma)|}2 |\tau_\sigma - M\xi_\sigma|.
	\end{equation}
	Note that the right--hand side is not $0$ since $j(\xi_\sigma)$ does not have complex multiplication
	but $j(\tau_\sigma)$ does. The same argument tells us $j'(\xi_\sigma) \not= 0$.
	Write $M = \begin{smatrix}\alpha&\beta\\\gamma&\delta\end{smatrix}$. Then we get
	\begin{align*}
		|\tau_\sigma - M\xi_\sigma| &= \left|\tau_\sigma - M\frac{\omega_{\sigma,2}}{\omega_{\sigma,1}}\right| 	\\
		&= \frac 1{|\gamma\omega_{\sigma,2}+\delta\omega_{\sigma,1}|}
				|(\gamma\tau-\alpha)\omega_{\sigma,2}+(\delta\tau-\beta)\omega_{\sigma,1}|	\\
				&\ge \frac 1{\max\{|\omega_{\sigma,2}|,|\omega_{\sigma,1}|\}}
				|(\gamma\tau-\alpha)\omega_{\sigma,2}+(\delta\tau-\beta)\omega_{\sigma,1}|
	\end{align*}
	since $\gamma$ and $\delta$ are not both $1$.
	Using this together with \eqref{eq:j_to_lin_form} and Lemma \ref{lem:lin_forms} we obtain
	\begin{align}
		\notag
		\log|\sigma(j-\alpha)| &= 
		\log|j(\tau_\sigma) - j(\xi_\sigma)|	\\
		\notag
		&\ge \log\frac{|j'(\xi_\sigma)|}2 + \log\min\{|\omega_{\sigma,1}|^{-1},|\omega_{\sigma,2}|^{-1}\}
							- c_2\left(\log|\Delta|\right)^4	\\
		&\ge \log\frac{|j'(\xi_\sigma)|}2 + \log\min\{1,|\omega_{\sigma,1}|^{-1},|\omega_{\sigma,2}|^{-1}\}
								- c_2\left(\log|\Delta|\right)^4.\label{eq:masser_lin_log}
	\end{align}
	Since with $|\sigma(j-\alpha)| < \varepsilon_0$ we also have $|\tau_\sigma - M\xi_\sigma| < \varepsilon$
	for some matrix $M \in \T$ as mentioned before,
	we obtain that $\tau_\sigma$ corresponds to a form in $\cm(\Delta; M\xi_\sigma; \varepsilon)$.
	We have
	\begin{displaymath}
		c(\xi_\sigma) \le \frac{|j'(\xi_\sigma)|}2 \delta \le \frac{|j'(\xi_\sigma)|}2
	\end{displaymath}
	so that
	\begin{displaymath}
		\log\max\left\{1,\frac 2{|j'(\xi_\sigma)|}\right\} \le \Pen(\xi).
	\end{displaymath}
	Thus, if we use \eqref{eq:masser_lin_log} we get
	\begin{align*}
		h(j-\alpha) \le& -\frac 1{[\Q(j,\alpha):\Q]}  \sum_{\lvert \sigma(j-\alpha) \rvert < \varepsilon_0} \log \lvert \sigma(j-\alpha) \rvert
					+ \lvert \log \varepsilon_0 \rvert	\\
			\le& c_2 \frac{\sum_{\sigma\colon \Q(\alpha)\hookrightarrow \C}\sum_{M \in \T}
									\cm(\Delta;M\xi_\sigma;\varepsilon)}{16\cdot[\Q(j,\alpha):\Q]}
										(\log|\Delta|)^4	\\
										&+ 4\Pen(\xi) - 4\log\min_\sigma\{1,|\omega_{\sigma,1}|^{-1},|\omega_{\sigma,2}|^{-1}\}
										+ |\log \varepsilon_0|.
	\end{align*}
	If we plug in the definition for $\varepsilon_0$ we obtain
	\begin{align*}
			h(j-\alpha) \le&c_2\frac{\sum_{\sigma\colon \Q(\alpha)\hookrightarrow \C}
														\sum_{M \in \T}\cm(\Delta;M\xi_\sigma;\varepsilon)}{16\cdot\cm(\Delta)}
										(\log|\Delta|)^4 + 4\Pen(\xi)	\\
										&+ 4\log\max_\sigma\{1,|\omega_{\sigma,1}|,|\omega_{\sigma,2}|\}
					- \log\min_\sigma\left\{1,c(\xi_\sigma)\right\}
					+ |\log \varepsilon|		\\
			=&c_2\frac{\sum_{\sigma\colon \Q(\alpha)\hookrightarrow \C}\sum_{M \in \T}\cm(\Delta;M\xi_\sigma;\varepsilon)}{16\cdot\cm(\Delta)}
										(\log|\Delta|)^4 + 5\Pen(\xi)	\\
										&+ 4\log\max_\sigma\{1,|\omega_{\sigma,1}|,|\omega_{\sigma,2}|\}
										+ |\log \varepsilon|,
	\end{align*}
	where we also used the first claim of Lemma \ref{lem:lower_trivial} for the inequality.
\end{proof}

\noindent The following lemmas can be found in Section 3 of \cite{bhk} as Lemmas 3.5 and 3.6.
\begin{lem}\label{lem:fbound}
	Assume that $|\Delta| \ge 10^{14}$. Then we have $F(\Delta) \ge |\Delta|^{0.34/\log\log(|\Delta|^{1/2})}$
	and $F(\Delta) \ge 18 \log\log(|\Delta|^{1/2})$.
\end{lem}

\begin{lem}\label{lem:cbound}
	For $\Delta \not= -3,-4$ we have $\cm(\Delta) \le \pi^{-1}|\Delta|^{1/2}(2+\log|\Delta|)$.
\end{lem}

\begin{proof}
	Theorem 10.1 in \cite{hua2012introduction} says $\cm(\Delta) \le \frac{\omega |\Delta|^{1/2}}{2\pi}K(d)$
	where $K(d)$ can be bounded by $2+\log|\Delta|$ according to Theorem 14.3 in \cite{hua2012introduction}
	and $\omega$ is the number of roots of unity in the imaginary quadratic order of discriminant
	$\Delta$. But since $\Delta \not= -3,-4$ we have $\omega =2$ and the result follows.
\end{proof}

\noindent We define $E := E(\Delta) = F(\Delta) (\log|\Delta|)^4$.

\begin{cor}\label{cor:upper_bound}
	Assume $j-\alpha$ is a unit and $\alpha = j(\xi)$ with $\xi \in \F$ and such that $\alpha$
	is algebraic but not a singular modulus.
	For $|\Delta| \ge 10^{14}$ we have
	\begin{displaymath}
		h(j-\alpha) \le c_2[\Q(\alpha):\Q]\frac{E(\Delta)}{2\cm(\Delta)} + \log\frac{E(\Delta)|\Delta|^{1/2}}{\cm(\Delta)} + C'
	\end{displaymath}
	where $C'$ is a constant depending on $\alpha$ and is given by
	\begin{align*}
		C' = 4[\Q(\alpha):\Q]c_2 + 5 \Pen(\xi) + 4\mathcal{M}(\xi).
	\end{align*}
\end{cor}

\begin{proof}
	We can use the previous results together with the bound for $\cm(\Delta; \xi_\sigma; \varepsilon)$ to bound the
	height $h(j-\alpha)$.
	Put $\varepsilon = \frac{\cm(\Delta)}{F(\Delta)(\log|\Delta|)^4 |\Delta|^{1/2}}$.
	Since we assume that $|\Delta| \ge 10^{14}$ we have $F(\Delta) \ge 256$ and we obtain together with Lemma \ref{lem:cbound}
	\begin{displaymath}
		\varepsilon \le \frac{\pi^{-1}|\Delta|^{1/2}(2+\log|\Delta|)}{256(\log|\Delta|)^4 |\Delta|^{1/2}}
		\le \frac 1{256\pi} \frac{2+\log 10^{14}}{\log 10^{14}} < 2\cdot 10^{-3}.
	\end{displaymath}
	Thus we can apply Proposition \ref{prop:ht_bound_w_sum} together with Corollary \ref{cor:eps_nbh} and obtain
	\begin{align*}
			h(j-\alpha)\le&[\Q(\alpha):\Q]c_2\frac{4F(\Delta)\left(32 |\Delta|^{1/2} \varepsilon^2 \log\log(|\Delta|^{1/2})
								+11\vert\Delta\vert^{1/2}\varepsilon + 2\right)
						}{16\cdot\cm(\Delta)}
										(\log|\Delta|)^4	\\
						&+ 5\Pen(\xi) + 4\mathcal{M}(\xi) + |\log \varepsilon|	\\
			=&[\Q(\alpha):\Q]c_2E\frac{128 |\Delta|^{1/2} \log\log(|\Delta|^{1/2})
						}{16\cdot\cm(\Delta)}
	 					\left(\frac{\cm(\Delta)}{F(\Delta)(\log|\Delta|)^4 |\Delta|^{1/2}}\right)^2	\\
			&+[\Q(\alpha):\Q]c_2E\frac{44\vert\Delta\vert^{1/2}
						}{16\cdot\cm(\Delta)}
	 					\frac{\cm(\Delta)}{F(\Delta)(\log|\Delta|)^4 |\Delta|^{1/2}}	\\
			&+[\Q(\alpha):\Q]c_2\frac{E}{2\cm(\Delta)} \\
			&+ 5\Pen(\xi) + 4\mathcal{M}(\xi)
			+ \log\left(\frac{F(\Delta)(\log|\Delta|)^4 |\Delta|^{1/2}}{\cm(\Delta)} \right).
	\end{align*}
	We continue the estimate by simplifying the terms to get
	\begin{align*}
			h(j-\alpha)\le&[\Q(\alpha):\Q]c_2\frac{8 \log\log(|\Delta|^{1/2})
						}{F(\Delta)}
	 					\frac{\cm(\Delta)}{(\log|\Delta|)^4 |\Delta|^{1/2}}	\\
			&+3[\Q(\alpha):\Q]c_2
			+[\Q(\alpha):\Q]c_2\frac{E}{2\cm(\Delta)} \\
			&+ 5\Pen(\xi) + 4\mathcal{M}(\xi) + \log\left(\frac{E|\Delta|^{1/2}}{\cm(\Delta)} \right).
	\end{align*}
	Now we apply Lemma \ref{lem:fbound} to see
	\begin{align*}
			h(j-\alpha)
			\le&[\Q(\alpha):\Q]c_2\frac{1}{2}
	 					\frac{\cm(\Delta)}{(\log|\Delta|)^4 |\Delta|^{1/2}}	\\
			&+3[\Q(\alpha):\Q]c_2
			+[\Q(\alpha):\Q]c_2\frac{E}{2\cm(\Delta)} \\
			&+5\Pen(\xi) + 4\mathcal{M}(\xi)
			+ \log\left(\frac{E|\Delta|^{1/2}}{\cm(\Delta)} \right).
	\end{align*}
	We continue the estimate using Lemma \ref{lem:cbound}
	\begin{align*}
			h(j-\alpha)\le&[\Q(\alpha):\Q]c_2\frac{1}{2\pi}
	 					\frac{|\Delta|^{1/2}(2+\log|\Delta|)}{(\log|\Delta|)^4 |\Delta|^{1/2}}	\\
			&+3[\Q(\alpha):\Q]c_2
			+[\Q(\alpha):\Q]c_2\frac{E}{2\cm(\Delta)} \\
			&+5\Pen(\xi) + 4\mathcal{M}(\xi)
			+ \log\left(\frac{E|\Delta|^{1/2}}{\cm(\Delta)} \right).
	\end{align*}
	Simplifying again results into
	\begin{align*}
			h(j-\alpha)\le&[\Q(\alpha):\Q]c_2\frac{1}{2\pi}
	 					\frac{(2+\log|\Delta|)}{(\log|\Delta|)^4}	\\
			&+3[\Q(\alpha):\Q]c_2
			+[\Q(\alpha):\Q]c_2\frac{E}{2\cm(\Delta)} \\
			&+5\Pen(\xi) + 4\mathcal{M}(\xi)
			+ \log\left(\frac{E|\Delta|^{1/2}}{\cm(\Delta)} \right)
	\end{align*}
	The function $x \mapsto \frac{2+\log x}{(\log x)^4}$ is decreasing for $x > 1$.
	Thus we can substitute $x = 10^{14}$ to continue the bound and get
	\begin{align*}
			h(j-\alpha)\le&[\Q(\alpha):\Q]c_2\frac{1}{2\pi}
	 					\frac{35}{32^4}	\\
			&+3[\Q(\alpha):\Q]c_2
			+[\Q(\alpha):\Q]c_2\frac{E}{2\cm(\Delta)} \\
			&+5\Pen(\xi) + 4\mathcal{M}(\xi)
			+ \log\left(\frac{E|\Delta|^{1/2}}{\cm(\Delta)} \right).
	\end{align*}
	This gives the desired inequality.
\end{proof}

\section{Proof of the main theorem}

We now want to bound $\Delta$ to complete the main proof.
We will do this by using the lower and upper bounds we derived in the last section.
Throughout this section we assume $|\Delta| \ge 10^{50}$.

Put
\begin{equation}\label{eq:const_C}
	C = C' + h(\alpha) + \log 2 + 0.01.
\end{equation}
Combining the lower bounds for $h(j-\alpha)$ from Lemmas \ref{lem:lower_trivial} and \ref{lem:lower_hard}
with the upper bound from Corollary \ref{cor:upper_bound} we obtain
the inequality
\begin{displaymath}
	L:=\max\left\{\pi\frac{|\Delta|^{1/2}}{\cm(\Delta)}, \frac 3{\sqrt5}\log|\Delta| - 10 \right\}
	\le [\Q(\alpha):\Q]c_2\frac{E(\Delta)}{2\cm(\Delta)} + \log\frac{E(\Delta)|\Delta|^{1/2}}{\cm(\Delta)} + C
\end{displaymath}
or equivalently
\begin{displaymath}
	1 \le [\Q(\alpha):\Q]c_2\frac E{2L\cdot\cm(\Delta)} + \frac{\log E + C}L + \frac{\log(|\Delta|^{1/2}/\cm(\Delta))}L.
\end{displaymath}
For the remainder we assume that $|\Delta|$ is large enough so that $\log E + C \ge 0$.
By Lemma \ref{lem:fbound} this is the case when $|\Delta| \ge e^{e^{e^{-C}/18}}$.
This in turn is true whenever $|\Delta| \ge 3$.
Since $\frac 3{\sqrt{5}}\log |\Delta| - 10 > 0$ for $|\Delta| \ge 10^{14}$
this allows us to replace $L$ by $\frac 3{\sqrt{5}}\log |\Delta| - 10$ in the middle term. Similarly we can replace $L$
in the first term by $\pi|\Delta|^{1/2}/\cm(\Delta)$ and obtain
\begin{align}
	1 &\le [\Q(\alpha):\Q]c_2\frac E{2\pi |\Delta|^{1/2}} + \frac{\log E + C}{\frac 3{\sqrt5}\log |\Delta| -10}
			+ \frac{\log(|\Delta|^{1/2}/\cm(\Delta))}L\nonumber	\\
			&\le [\Q(\alpha):\Q]c_2\frac E{2\pi |\Delta|^{1/2}} + \frac{\log E}{\frac 3{\sqrt{5}}\log |\Delta| -10}
			+ \frac{\log(\pi^{-1}L)}L + \frac{C}{\frac 3{\sqrt5}\log |\Delta| -10}\label{eq:smaller1}
\end{align}

We want to show that the right--hand side is less than 1 for large enough $|\Delta|$. Before we start, we want to give
a bound on $E(\Delta) = F(\Delta)(\log|\Delta|)^4$. To do this, we are going to bound $\log F(\Delta)$ and $\log E(\Delta)$.
By Th\'eor\`eme 1.1 in \cite{robin1983estimation} we have
\begin{displaymath}
	\omega(n) \le 1.4 \frac{\log n}{\log\log n}
\end{displaymath}
for any $n \ge 3$.
Therefore we obtain the bound
\begin{displaymath}
	\frac{\log F(\Delta)}{\log 2} \le 1.4 \frac{\log |\Delta|}{\log\log|\Delta|}.
\end{displaymath}
Then the bound on $\log E(\Delta)$ is given by
\begin{displaymath}
	\log E(\Delta) \le \frac{\log|\Delta|}{\log\log|\Delta|} + 4\log\log|\Delta|.
\end{displaymath}

Now we want to bound $E|\Delta|^{-1/2}$. Since the function
\begin{displaymath}
	u_0(x) = \frac{1}{\log\log x} + \frac{4\log\log x}{\log x} - \frac 12
\end{displaymath}
is decreasing for $x \ge 10^{10}$ we obtain
\begin{displaymath}
	\frac{\log(E|\Delta|^{-1/2})}{\log|\Delta|} \le u_0(10^{50}) < -\frac 1{10}
\end{displaymath}
for $|\Delta| \ge 10^{50}$.
This in turn implies
\begin{equation}\label{eq:first}
	E |\Delta|^{-1/2} < |\Delta|^{-0.1}.
\end{equation}

The next step is to bound the second term of \eqref{eq:smaller1}. The functions
\begin{displaymath}
	u_1(x) = \log 2 \frac 1{\log\log x - c_1 - \log 2} + 4\frac{\log\log x}{\log x}
\end{displaymath}
and
\begin{displaymath}
	u_2(x) = \left( \frac 3{\sqrt{5}} - \frac{10}{\log x} \right)^{-1}
\end{displaymath}
are decreasing for $x \ge 10^{10}$. We have
\begin{equation}\label{eq:second}
	\frac{\log E}{\frac 3{\sqrt5}\log |\Delta| -10} \le u_1(|\Delta|) u_2(|\Delta|) \le u_1(10^{50}) u_2(10^{50})
	\le 0.4896
\end{equation}
for $|\Delta| \ge 10^{50}$.

To bound the third term of \eqref{eq:smaller1} we remark that the function $x \mapsto x^{-1}\log(\pi^{-1}x)$
is decreasing for $x \ge e/\pi$.
We have $L \ge \frac 3{\sqrt{5}}\log |\Delta| - 10 \ge e/\pi$ for $\Delta \ge 10^{15}$ and therefore
\begin{displaymath}
	\frac{\log(\pi^{-1}L)}L
	\le \frac{\log\left(\pi^{-1}\left(\frac{3}{\sqrt{5}}\log|\Delta| - 10\right)\right)}{\frac{3}{\sqrt{5}}\log|\Delta| - 10}.
\end{displaymath}
The function
\begin{displaymath}
	u_3(x) := \frac{\log\left(\pi^{-1}\left(\frac{3}{\sqrt{5}}\log x - 10\right)\right)}{\frac{3}{\sqrt{5}}\log x - 10}
\end{displaymath}
is decreasing for $x \ge 10^{15}$. Thus we obtain
\begin{equation}\label{eq:third}
	\frac{\log(\pi^{-1}L)}L \le u_3(|\Delta|) \le u_3(10^{15}) < 0.0674 < \frac 1{10}
\end{equation}
for $|\Delta| \ge 10^{50}$.

For $|\Delta| \ge e^{10\frac{\sqrt{5}}{3}(C+1)}$ we have
\begin{equation}
	\frac{C}{\frac 3{\sqrt5}\log |\Delta| -10} \le \frac 1{10}.
\end{equation}
By equation \eqref{eq:first} we can bound the first term of \eqref{eq:smaller1} by
\begin{displaymath}
	\left(\frac{[\Q(\alpha):\Q]c_2}{2\pi}\right) |\Delta|^{-0.1} \le \frac 1{10}
\end{displaymath}
for $|\Delta| \ge (10[\Q(\alpha):\Q]c_2/(2\pi))^{10}$.

Using these two inequalities together with \eqref{eq:second} and \eqref{eq:third} we obtain that
\eqref{eq:smaller1} is less than $1$ for all
\begin{equation}\label{eq:delta_bound_max}
	|\Delta| \ge \max\left\{10^{50}, e^{10\frac{\sqrt{5}}{3}(C+1)},
	(10[\Q(\alpha):\Q]c_2/(4\pi))^{10}\right\}.
\end{equation}
This contradicts the lower bound of \eqref{eq:smaller1}.	\\
The lower bound on $|\Delta|$ can be simplified.
In equation \eqref{eq:const_C} we have put $C = 4[\Q(\alpha):\Q]c_2 + 5 \Pen(\xi) + 4\mathcal{M}(\xi) + h(\alpha) + \log 2 + 0.01$.
Recall that $c_2 \ge 1$. This implies $C \ge 13$ and hence
\begin{displaymath}
	e^{15C} \ge e^{195} \ge 10^{50}.
\end{displaymath}
Moreover, we have
\begin{align*}
	15C \ge {10\frac{\sqrt{5}}{3}(C+1)} &\ge {10\frac{\sqrt{5}}{3}2[\Q(\alpha):\Q]c_2}
	\ge {10\frac{5}{4\pi}2[\Q(\alpha):\Q]c_2}	\\
	&\ge {10\log\left(\frac{5}{4\pi}2[\Q(\alpha):\Q]c_2\right)}.
\end{align*}
Therefore, the bound on $|\Delta|$ from equation \eqref{eq:delta_bound_max} simplifies to
\begin{equation}\label{eq:delta_bound}
	|\Delta| \ge e^{15C},
\end{equation}
where $C > 0$ is the (computable) constant
\begin{displaymath}
	C = 2[\Q(\alpha):\Q]c_2 + 6 \Pen(\xi) + 4\mathcal{M}(\xi) + h(\alpha) + \log 2 + 0.01.
\end{displaymath}

\vspace{3cm}
\addcontentsline{toc}{chapter}{\bibname}
\thispagestyle{plain}
\printbibliography

\end{document}